\documentclass[reqno,a4paper,11pt]{amsart}
\usepackage{amsmath,amstext,amsfonts,amsbsy,eucal,amssymb}
\usepackage[latin1]{inputenc}

\numberwithin{equation}{section}

\newtheorem{theorem}{Theorem}[section]

\newtheorem{proposition}[theorem]{Proposition}

\theoremstyle{definition}

\newtheorem{remark}[theorem]{Remark}

\begin{document}

\parskip 4pt
\baselineskip 16pt

%%%%%%%%%%%%%%%%%%%%%%%%%%%%%%%%%
%%%%%%%%%%%%%%%%%%%%%%%%%%%%%%%%%

\title[On some sequences of polynomials generating the Genocchi numbers]
{On some sequences of polynomials generating the Genocchi numbers}
\author[Andrei K. Svinin]{Andrei K. Svinin}
\address{Andrei K. Svinin, 
Matrosov Institute for System Dynamics and Control Theory of 
Siberian Branch of Russian Academy of Sciences,
P.O. Box 292, 664033 Irkutsk, Russia}
\email{svinin@icc.ru}
\date{\today}

\begin{abstract}
Sequences of Genocchi numbers of the first and second kind are considered. For these numbers, an approach  based on their representation using sequences of polynomials is developed. Based on this approach, for these numbers some identities generalizing the known identities are constructed.
\end{abstract}
\maketitle

\section{Introduction} 

 Classical Genocchi numbers $(G_{2n})_{n\geq 1}=(1, 1, 3, 17, 155,\ldots)$  or, as they say, Genocchi numbers of the first kind,
most easily defined by an exponential generating function (see for example \cite{Comtet1}, \cite{Stanley})
\begin{equation}
\frac{2t}{e^t+1}=t+\sum_{n\geq 1}(-1)^nG_{2n}\frac{t^{2n}}{(2n)!}.
\label{egf}
\end{equation}
These numbers have a big number of applications in different areas of mathematics when enumerating various objects. The most famous applications of these numbers are that they, like Euler numbers, enumerate permutations of a given type \cite{Dumont4}. Several works by Dumont et al. (see, for example, \cite {Dumont3}, \cite{Dumont2}, \cite {Dumont}) and subsequent works by other authors are devoted to this line.

It follows from the generating function (\ref{egf}) that the numbers $G_{2n}$ are related to the Bernoulli numbers by the relation $G_{2n}=(-1)^{n+1}2(4^n-1)B_{2n}$, that, in turn, are determined by the relation
\[
\frac{t}{e^t-1}=1-\frac{t}{2}+\sum_{n\geq 1}B_{2n}\frac{t^{2n}}{(2n)!}.
\] 
The Genocchi numbers are contained in some numerical triangles. The most classic example is the Seidel triangle, which consists of the numbers $g_{n, j}$, where $j\geq 1$ and $1\leq n\leq (j+1)/2$, defined as follows. Let $g_{1, 1}=1$, and the remaining numbers in the triangle are determined by the formulas
\[
g_{n, 2j}=\sum_{q\geq n} g_{q, 2j-1},\;\;
g_{n, 2j+1}=\sum_{q\leq n} g_{q, 2j}.
\]
In the Seidel's triangle, Genocchi numbers are defined as $G_{2n}=g_{n, 2n-1}$. Also, in the Seidel triangle, one can find the so-called median  Genocchi  numbers or the second kind Genocchi numbers $(H_{2n-1})_{n\geq 1}=(1, 2, 8, 56,\ldots)$, namely $H_{2n-1}=g_ {1, 2n}$. These numbers, as it is known, also have good applications when enumerating various mathematical objects. Barsky \cite{Barsky} and Dumont \cite{Dumont3} proved that the number $H_{2n + 1}$ is divisible by $2^n$ for any $n\geq 0$ and, therefore, the number $h_n=H_{2n+1}/2^n$ for $n\geq 0$ is an integer. The numbers $(h_n)_{n\geq 0}$ are called the normalized median Genocchi numbers. A good exposition of the combinatorial meaning of these numbers can be found in \cite{Feigin}.

Speaking of the Genocchi numbers, we should mention the classical identities involving these numbers. The best known identity is Seidel's implicit recurrence relation
\[
\sum_{j=0}^{\left\lfloor\frac{n}{2} \right\rfloor }(-1)^j{n\choose 2j}G_{2n-2j}=0\;\;\forall n\geq 2.
\]
It is known that the Genocchi numbers of the first and second kind are related by
\[
H_{2n-1}=\sum_{j=0}^{\left\lfloor\frac{n-1}{2} \right\rfloor }(-1)^{j}{n\choose 2j+1}G_{2n-2j}\;\; \forall n\geq 1. 
\]

One of the ways to generate Genocchi numbers of both the first and second kind is related to the fact that they can be given as values of a sequence of polynomials determined by a recurrence relation for some value of the argument. The purpose of the article is to develop and apply this approach. Stirling type numbers are involved in the construction of these polynomials; therefore, in the next section, we briefly outline the approach to constructing these numbers. In the  section \ref{sect:3}, we give two examples of sequences of polynomials determined by the corresponding numbers of Stirling type and show how they participate in the construction of some identities, including the Genocchi numbers. Thus, one of the results of the article is the construction of new identities on Genocchi numbers that generalize known identities. In the \ref{sect:4} section, we present a countable class of sequences of polynomials that uniformly and naturally define a countable class of integer sequences. The first two sequences in this class are the Genocchi numbers of the first and second kind, respectively. The rest of the integer sequences belonging to this class we call Genocchi type numbers.

\section{Stirling type numbers} 

First of all, let us say about the general approach for constructing Stirling type numbers \cite{Comtet}, in the framework of which two-parameter sets of numbers of interest are determined. Given an arbitrary sequence of numbers $(a_n)_{n\geq 1}$, we define two two-parameter sets of polynomials in the variables $(a_1, a_2,\ldots)$ by the relations
\begin{equation}
t(n+1,j)=t(n,j-1)-a_nt(n,j)
\label{first}
\end{equation}
and
\begin{equation}
T(n+1,j)=T(n,j-1)+a_jT(n,j)
\label{second}
\end{equation}
with the condition that $T (n, 0)=t (n, 0)=\delta_{n, 0}$ and $t(n, j)=T (n, j)=0$ for $n<j $. It turns out that $T(n, j)$ and $t(n, j)$ are symmetric polynomials in their variables. More precisely, they can be calculated using generating functions
\[
t(t-a_1)\cdots(t-a_{n-1})=\sum_{j=0}^{n}t(n, j)t^j\;\;\mbox{and}\;\;
\sum_{n\geq j}T(n, j)t^{n-j}=\frac{1}{1-a_1t}\cdots\frac{1}{1-a_jt}.
\]
Given a  suitable sequence of numbers $(a_n)_{n\geq 1}$, we get some numbers of Stirling type of the first and second kind, respectively. For example, the choice of $a_n=n$ corresponds to the classic Stirling numbers $s (n, j)$ and $S(n, j)$. In the case  $a_n=n^2$, we get the so-called central factorial numbers \cite {Carlitz3}, \cite{Stanley}, which we  denote as in the general case by the symbols $t(n, j)$ and $T(n, j)$. In turn, if $a_n=n(n+1)$, then we get the so-called Legendre-Stirling numbers \cite{Andrews}, which we denote as $Ls(n, j)$ and $LS(n, j )$.

\section{Sequences of polynomials generating the Genocchi numbers}

\label{sect:3}

\subsection{Gandhi polynomials}

A common property of the first and second kind Genocchi numbers is that they can be obtained as the values of the corresponding polynomials $F_n(z)$ for some value of the argument. These polynomials are determined by a recurrence relation of the form
\[
F_{n+1}(z)=f(z)F_{n}(z+1)-g(z)F_{n}(z)
\]
starting from $F_1(z)=1$, where $f(z)$ and $g(z)$ are some suitable second-degree polynomials. For example, the numbers $G_{2n}$ are defined by the Gandhi polynomials \cite{Gandhi}, \cite{Carlitz2}, \cite {Riordan}, which are given by the recurrence relation
\begin{equation}
F_{n+1}(z)=(z+1)^2F_{n}(z+1)-z^2F_{n}(z).
\label{recurG}
\end{equation}
The first few of them are as follows:
\[
F_2(z)=2z+1,\;\;
F_3(z)=6z^2+8z+3,\;\;
F_4(z)=24z^3+60z^2+54z+17,
\]
\[
F_5(z)=120z^4+480z^3+762z^2+556z+155,\ldots
\]
In the works \cite{Carlitz2}, \cite{Riordan} it was proved that the Gandhi polynomials are related to the first-kind Genocchi numbers by the relations $F_n(0)=G_{2n}$ and $F_n(1)=G_{2n+2}$ for all $n \geq 1$.
\begin{remark}
Gandhi in \cite{Gandhi} actually defined the polynomials $A_n(z)$
\[
A_{n+1}(z)=z^2A_{n}(z+1)-(z-1)^2A_{n}(z),
\]
starting from $A_0(z)=1$, which are related to $F_n (z)$ by the relation $F_n(z)=A_{n-1}(z+1)$. The polynomials $F_n(z)$ defined by the relation (\ref{recurG}) were used, for example, in \cite{Dumont}.
\end{remark}
It is easy to show \cite{Carlitz} that the Gandhi polynomials can be represented as
\begin{equation}
zF_{n}(z)=\sum_{j=0}^n(-1)^{n+j}T(n, j)j! (z)^{j},
\label{polynom}
\end{equation}
where $(z)^0:=1$ and $(z)^n:=z(z+1)\cdots (z+n-1)\;\; \forall n\geq 1 $, and $T(n, j)$ are the central factorial numbers of the second kind mentioned above. To prove the representation (\ref{polynom}), it is more convenient to consider polynomials $P_n(z):=zF_n(z) $ satisfying, by virtue of (\ref{recurG}), the recurrence relation
\[
P_{n+1}(z)=z(z+1)P_{n}(z+1)-z^2P_{n}(z)
\]
starting from $P_1(z)=z$, which can be written in the operator form $ P_{n+1}(z)=R\left(P_{n} (z)\right) $ with $ R:=z(z+1) \Lambda-z^2 $, where $\Lambda$ is the shift operator acting according to the rule $\Lambda(f(z))=f(z+1)$. To prove the relation (\ref{polynom}), one can use the easily verified relation
\[
R((z)^n)=-n^2(z)^n+(n+1)(z)^{n+1}\;\;\forall n\geq 0
\] 
and a recurrence relation for central factorial numbers of the second kind.  

In turn, given that $F_n(1)=G_{2n+2}$, it is not difficult to deduce from the relation (\ref{polynom}) the well-known identity \cite{Dumont3}, \cite {Carlitz}, \cite{Stanley}
\begin{equation}
G_{2n+2}=\sum_{j=0}^n(-1)^{n+j}T(n, j)\left(j!\right)^2\;\;\forall n\geq 1.
\label{polynom1}
\end{equation}
Next, we derive the identity for the numbers $G_{2n}$ generalizing (\ref{polynom1}).
\begin{proposition}
The values of the Gandhi polynomials for integer non-negative values of the argument are determined by the relation
\begin{equation}
F_{n}(m)(m!)^2=\sum_{j=0}^{m}(-1)^{m+j}t(m, j)G_{2n+2j},
\label{Fnom121}
\end{equation}
where $t(n, j)$ are the central factorial numbers of the first kind
\end{proposition}
We do not prove this proposition here, since it will be proved later for a more general case.

In turn, it follows from (\ref{polynom}) that
\begin{equation}
F_{n}(m)=\frac{1}{m!}\sum_{j=0}^n(-1)^{n+j}T(n, j)j!(j+m-1)!.
\label{Fnom1212}
\end{equation}
Comparing (\ref{Fnom121}) and (\ref{Fnom1212}), we obtain the following result:
\begin{proposition}
The following identity holds:
\begin{equation}
\sum_{j=0}^{m}(-1)^{m+j}t(m, j)G_{2n+2j}=m!\sum_{j=0}^n(-1)^{n+j}T(n, j)j!(j+m-1)!\;\; \forall m\geq 0
\label{Fnom121234}
\end{equation}
for all $m\geq 0$ and $n\geq 1$.
\end{proposition}
\begin{remark}
At first glance, for $m=0$, the indefinite symbol $(-1)!$ appears in the formula (\ref{Fnom121234}), but since $T(n, 0)=0$ for all $n\geq 1$, then in fact in this case the formula is correct and we can write it as
\[
G_{2n}=\sum_{j=1}^n(-1)^{n+j}T(n, j)(j-1)!j!\;\;\forall n\geq 1. 
\]
\end{remark}

\subsection{Polynomials generating Genocchi numbers of the second kind}

In \cite{Dumont} (see also \cite{Claesson}), a sequence of polynomials $(C_n(z))_{n\geq 1}$ was defined by the recurrence relation
\[
C_{n+1}(z)=(z+1)^2C_{n}(z+1)-z(z+1)C_{n}(z)
\]
which starts from $C_1(z)= 1$ and it was proved that $C_n(1)=H_{2n+1}\;\;\forall n\geq 1$. It is easy to show that there exist polynomials $\tilde{F}_n(z) $ connected with $C_n(z)$ by the relation $C_n(z)=(z+1)\tilde{F}_{n-1}(z+1)$ for all $n\geq 2$. It is easy to verify that these polynomials are determined by the recurrence relation
\begin{equation}
\tilde{F}_{n+1}(z)=z(z+1)\tilde{F}_{n}(z+1)-(z-1)z\tilde{F}_{n}(z)
\label{rrel}
\end{equation}
starting with $\tilde{F}_1(z)=1$. The first few polynomials are as follows:
\[
\tilde{F}_2(z)=2z,\;\;
\tilde{F}_3(z)=6z^2+2z,\;\;
\tilde{F}_4(z)=24z^3+24z^2+8z,
\]
\[
\tilde{F}_5(z)=120z^4+240z^3+192z^2+56z,\ldots
\]
It easily follows from the relation (\ref{rrel}) that $\tilde{F}_n(0)=0\;\; \forall n \geq 2$. If $C_n(1)=H_{2n+1}$, then it is obvious that $\tilde{F}_n(2)=H_{2n+1}/2$. In turn, using the recurrence relation (\ref{rrel}), we can calculate the values of the polynomials $\tilde{F}_n(z) $ for any natural value of the argument $z$. For example, we can write
\[
\tilde{F}_n(1)=\frac{H_{2n-1}}{0!1!},\;\;
\tilde{F}_n(2)=\frac{H_{2n+1}}{1!2!},\;\;
\tilde{F}_n(3)=\frac{2H_{2n+1}+H_{2n+3}}{2!3!},\;\;
\]
\[
\tilde{F}_n(4)=\frac{12H_{2n+1}+8H_{2n+3}+H_{2n+5}}{3!4!},\ldots
\]
for all $n\geq 1$. We can give a general formula for the values of polynomials defined by the recurrence relation (\ref{rrel}).
\begin{proposition} \label{pr:1} 
The values of the polynomials $\tilde{F}_{n} (z) $ for integer non-negative values of the argument are determined by the relation
\begin{equation}
\sum_{j=0}^{m}(-1)^{m+j}Ls(m, j)H_{2n+2j-1}=\tilde{F}_{n}(m+1)m!(m+1)!,
\label{Fnom1}
\end{equation}
where $Ls(n, j)$  are Legendre-Stirling numbers of the first kind.  
\end{proposition}
We also will not prove this proposition, since it is a special case of a more general statement, the proof of which will be presented later.

\subsection{Representation of Genocchi numbers of the second kind through Legendre-Stirling numbers}

Apparently, in \cite{Claesson} was obtained the relation 
\begin{equation}
H_{2n-1}=\sum_{j=0}^n(-1)^{n+j}LS(n, j)\left(j!\right)^2\forall n\geq 0,
\label{getthe}
\end{equation}
which is similar to the representation (\ref{polynom1}) for the first kind Genocchi numbers with the only difference that the numbers
$T(n, j)$ here are replaced by Legendre-Stirling numbers.  Since the representation (\ref{polynom1}) is deduced from the corresponding representation for the Gandhi polynomials (\ref{polynom}), it is natural to assume that a similar representation also exists for the polynomials $\tilde{F}_n(z) $ and indeed it is.
\begin{proposition}
The polynomials $ \tilde{F}_n(z)$ are representable in the form
\begin{equation}
z\tilde{F}_{n}(z)=\sum_{j=0}^n(-1)^{n+j}LS(n, j)j! (z)^{j}.
\label{polynom11}
\end{equation}
\end{proposition}
\begin{proof}
The proof of the representation (\ref{polynom11}) is similar to the proof of (\ref{polynom}). It is more convenient to use the polynomials $\tilde{P}_n(z):=z\tilde{F}_n(z)$, which, as can be easily verified, satisfy the recurrence relation
\begin{equation}
\tilde{P}_{n+1}(z)=z^2\tilde{P}_{n}(z+1)-(z-1)z\tilde{P}_{n}(z)
\label{rrel11}
\end{equation}
starting from $\tilde{P}_1(z)=z$. We rewrite this relation in the operator form $\tilde{P}_{n+1}(z)=\tilde{R}\left(\tilde{P}_{n}(z)\right) $ with the operator $\tilde{R}=z^2\Lambda-(z-1)z$. It is easy to verify that
\[
\tilde{R}\left((z)^n\right)=-n(n+1)(z)^n+(n+1)(z)^{n+1}\;\; \forall n\geq 0.
\]
\end{proof}

The relation (\ref{polynom11}) implies
\[
\tilde{F}_{n}(m+1)=\frac{1}{(m+1)!}\sum_{j=0}^n(-1)^{n+j}LS(n, j)j! (j+m)!.
\]
Comparing this expression with (\ref{Fnom1}), we get
\begin{proposition}
The identity
\begin{equation}
\sum_{j=0}^{m}(-1)^{m+j}Ls(m, j)H_{2n+2j-1}=m!\sum_{j=0}^n(-1)^{n+j}LS(n, j)j! (j+m)!
\label{ident}
\end{equation}
is valid for all $m\geq 0$ and $n\geq 1$.
\end{proposition}
As a result, we have an identity generalizing (\ref{getthe}).

\section{General approach} 

\label{sect:4}

\subsection{Polynomials determined by Stirling type numbers}

So, we have at our disposal two substantial examples of sequences of polynomials $F_n(z)$ defined by the formula
\begin{equation}
zF_{n}(z)=\sum_{j=0}^n(-1)^{n+j}T(n, j)j! (z)^{j},
\label{Pnz21}
\end{equation}
where $T(n, j)$ are Stirling type numbers of the second kind corresponding to some sequence $(a_n)_{n\geq 1}$ and satisfying the relation (\ref{second}). Note that for these two examples the sequence of polynomials satisfies a recurrence relation of the form
\begin{equation}
F_{n+1}(z)=h(z)F_{n}(z+1)-g(z)F_{n}(z),
\label{Pnz11}
\end{equation}
starting from $F_1(z)=1$, where $h(z)$ and $g(z)$ are some polynomials of the second degree, although in the general case  a sequence of polynomials of the form (\ref{Pnz21}) not required to satisfy any  relations. Taking into account that
\[
(z)^{n}=\sum_{j=0}^n(-1)^{n+j}s(n, j)z^j,
\]
where $s(n, j)$ are the classical Stirling numbers of the first kind, we get
\begin{equation}
F_{n}(z)=\sum_{j=1}^n(-1)^{n+j}\left(\sum_{q=j}^nq!s(q, j)T(n, q)\right)z^{j-1}.
\label{genrep}
\end{equation}
For example, we have
\[
F_1(z)=1,\;\;F_2(z)= 2z-\left(T(2, 1)-2\right),
\]
\[
F_3(z)= 6z^2-\left(2T(3, 2)-18\right)z+\left(T(3, 1)-2T(3, 2)+12\right),
\]
\begin{eqnarray}
F_4(z)&=& 24z^3-\left(6T(4, 3)-144\right)z^2+(2T(4, 2)-18T(4, 3)+264)z \nonumber\\
&&-\left(T(4, 1)-2T(4, 2)+12T(4, 3)+144\right),\ldots \nonumber
\end{eqnarray}
Here we take into account that $T(n, n)=1\;\; \forall n\geq 1$.

\subsection{Representation for the Genocchi numbers of the second kind}

We make a digression and note that using  (\ref{genrep}), we can obtain the following representation for the second kind Genocchi numbers:
\begin{proposition}
The identity
\[
H_{2n-3}=(-1)^n\sum_{j=2}^nLS(n, j)j!s(j, 2)-\delta_{n, 2}
\]
holds for all   $n\geq 2$.
\end{proposition}
\begin{proof}
Consider the sequence of polynomials $\tilde{F}_n(z)$ corresponding to the number sequence $a_n=n(n+1)$. As we know, these polynomials satisfy the recurrence relation (\ref{rrel}). From this relation we get
\[
\tilde{F}^{\prime}_{n+1}(z)=\left(2z+1\right)\tilde{F}_{n}(z+1)+z(z+1)\tilde{F}^{\prime}_{n}(z+1)-(2z-1)\tilde{F}_{n}(z)-(z-1)z\tilde{F}^{\prime}_{n}(z).
\]
Hence 
\[
\tilde{F}^{\prime}_{n}(0)=\tilde{F}_{n-1}(1)+\tilde{F}_{n-1}(0)=H_{2n-3}+\delta_{n, 2}.
\]
In turn, from  (\ref{genrep}) we get
\[
\tilde{F}^{\prime}_{n}(0)=\sum_{j=2}^nLS(n, j)j!s(j, 2).
\]
Thus, the proposition is proved.
\end{proof}

\subsection{Recurrence relation}

Consider the operator $R:=f(z)\Lambda -g(z)$, where $f(z)$ and $ g(z)$ are supposed to be some second-order polynomials 
\begin{proposition} \label{pr:6}
Let $(a_n)_{n\geq 1}$ be an arbitrary numerical sequence. The relation
\begin{equation}
R\left((z)^n\right)=-a_n(z)^n+(n+1)(z)^{n+1}\;\; \forall n\geq 1
\label{Rrec}
\end{equation}
is true if and only if
\begin{equation}
f(z)=z(z+4+a_1-a_2),\;\; g(z)=z^2+(3+a_1-a_2)z+2+2a_1-a_2
\label{1}
\end{equation}
and
\begin{equation}
a_n= -a_1(n-2)+a_2(n-1)+(n-2)(n-1)\;\; \forall n\geq 3.
\label{2}
\end{equation}
\end{proposition}
\begin{proof}
On the one hand, it is easy to verify that, by substituting (\ref{1}) and (\ref{2}) into the relation (\ref{Rrec}), we obtain the identity for any $n\geq 1$. On the other hand, let $f(z)=f_2z^2+f_1z+f_0$ and $g(z)=g_2z^2+g_1z+g_0$, where $(f_2, f_1, f_0, g_2, g_1, g_0) $ are six undetermined coefficients. Comparing the coefficients at equal powers of $z$ in (\ref{Rrec}) for $n = 1, 2 $, we obtain a system of linear equations for these coefficients from which we uniquely determine $f_2=1,\; f_1=4+a_1-a_2,\; f_0=0$ and $g_2 = 1,\; g_1=3+a_1-a_2,\; g_0=2+2a_1-a_2 $. Clearly, the relation (\ref {Rrec}) for $n\geq 3$ uniquely determines $a_n$ for $n\geq 3 $ and we already know that these values are given by the expression (\ref{2}).
\end{proof}

By direct calculation it is easy to verify that, by virtue of (\ref{Rrec}), the sequence of polynomials $P_n(z)=zF_n(z)$ defined by the formula (\ref{Pnz21}) satisfies the recurrence relation
\begin{equation}
P_{n+1}(z)=f(z)P_{n}(z+1)-g(z)P_{n}(z).
\label{4365}
\end{equation}
with $P_1(z)=z$, where the polynomials $f(z)$ and $g(z) $ are determined by the formula (\ref{1}). The case $(a_1, a_2)=(1, 4)$ corresponds to the Gandhi polynomials. In turn, if $(a_1, a_2) = (2, 6)$, then we get a sequence of polynomials defined by the relation (\ref{rrel}). Note that in both cases the relation $a_2 = 2a_1 + 2$ holds, and as a result, $f(z)-g(z) = z$. In this case, the calculation of the polynomials $P_n(z)$ using the recurrence relation (\ref{4365}) can be performed starting from $P_0(z)=1$.

\subsection{Genocchi type numbers}

We now consider a countable class of number sequences $a_n^{(v)}=n(n+v)$ for integers $v\geq 0$. Using the formulas (\ref{first}) and (\ref{second}), for each such sequence we calculate the corresponding Stirling type numbers $t^{(v)}(m, j)$ and $T^{(v) }(m, j)$ and then we determine corresponding  sequence of polynomials $(F_n^{(v)}(z))_{n\geq 1}$ by the formula
\begin{equation}
zF_{n}^{(v)}(z)=\sum_{j=0}^n(-1)^{n+j}T^{(v)}(n, j)j! (z)^{j}.
\label{mnogoch}
\end{equation}
It is easy to see that, by the proposition \ref{pr:6},  that the polynomials $F_{n}^{(v)}(z)$ can be defined by the recurrence relation 
\begin{equation}
F^{(v)}_{n+1}(z)=(z+1)\left(z-v+1\right)F^{(v)}_{n}(z+1)-z(z-v)F^{(v)}_{n}(z)
\label{recurels}
\end{equation}
starting from $F^{(v)}_{1}(z)$. In fact, these are polynomials of two variables $v$ and $z$, which can take any complex values. The first few of them are as follows
\[
F^{(v)}_2(z)=2z-v+1,\;\;F^{(v)}_3(z)=6z^2-\left(6v-8\right)z+v^2-4v+3,
\]
\[
F^{(v)}_4(z)=24z^3-\left(36v-60\right)z^2+\left(14v^2-60v+54\right)z-v^3+11v^2-27v+17,\ldots
\]
\begin{remark}
Obviously, the polynomials $ P^{(v)}_{n}(z)=zF^{(v)}_{n}(z)$ for $n\geq 1$ are defined by the recurrence relation
\[
P^{(v)}_{n+1}(z)=z(z-v+1)P^{(v)}_{n}(z+1)-z(z-v)P^{(v)}_{n}(z),
\]
starting from $P^{(v)}_1(z)=z$. Let $v$ be an arbitrary complex number. It is easy to verify that the sequence $a_n=n(n+v)$ gives the general case of a sequence of the form (\ref{2}) with the condition $a_2=2a_1+2$. In this case, $f(z)-g(z) = z$ and therefore these polynomials can be calculated starting from $ P^{(v)}_0(z) = 1$.
\end{remark}
Let us pay attention to the similarity of formulas expressing identities associated with the Genocchi numbers and the corresponding polynomials. We notice that, in principle, one can write the formulas (\ref{Fnom121})and (\ref{Fnom1})  in the form
\begin{equation}
\sum_{j=0}^m(-1)^{m+j}t^{(v)}(m, j)g_{n+j}^{(v)}=F_n^{(v)}(m+v)\frac{m!(m+v)!}{v!},
\label{general}
\end{equation}
for $v=0, 1$. Moreover, $g_{n}^{(0)}=G_{2n}$ and $g_{n}^{(1)}=H_{2n-1}$ for $ n\geq 1$. First of all, we note that if we suppose that (\ref{general}) is true, then the numbers $g_{n}^{(v)}$ for each fixed $v\geq 0$ are defined as
\begin{equation}
g_{n}^{(v)}=F_n^{(v)}(v)\;\;
\forall n\geq 1, v\geq 0.
\label{gnv1}
\end{equation}
\begin{proposition}
\label{pr:8}
For any integer $v\geq 0$, and the numbers defined by the relation (\ref{gnv1}), the formula (\ref{general}) is valid for all $m\geq 0$.
\end{proposition}
\begin{proof}
Thus, for each fixed $n\geq 1$, the numbers $g_{n}^{(v)}$ can be calculated as the values of the corresponding polynomial $g_{n}(v):=F_n^{(v)}(v)$ for the integer values of the argument $v\geq 0 $. We need to prove that for numbers defined in this way, the relation (\ref{general}) also holds for the remaining $m\geq 1 $.
Let us denote
\[
\xi_{m, n}^{(v)}=\sum_{j=0}^{m}(-1)^{m+j}t^{(v)}(m, j)g_{n+j}^{(v)}.
\]
The numbers $\xi_{m, n}^{(v)}$ satisfy the relation
\begin{equation}
\xi_{m+1, n}^{(v)}=\xi_{m, n+1}^{(v)}+m\left(m+v\right)\xi_{m, n}^{(v)},
\label{sootn}
\end{equation}
with the condition $\xi_{0, n}^{(u)}=g_n^{(v)}$. From the last relation, in particular, we obtain
\[
\xi_{1, n}^{(v)}=g_{n+1}^{(v)},\;\;
\xi_{2, n}^{(v)}=g_{n+2}^{(v)}+(v+1)g_{n+1}^{(v)},\ldots
\]
Let us notice that the relation (\ref{sootn}) is determined only by the fact that Stirling type numbers of the first kind $t^{(v)}(n, j)$, by definition, satisfy the recurrence relation
\[
t^{(v)}(n+1,j)=t^{(v)}(n,j-1)-n(n+v)t^{(v)}(n,j).
\]
Therefore, we must prove that
\[
\xi_{m, n}^{(v)}=F_n^{(v)}(m+v)\frac{m!(m+v)!}{v!}.
\]
Substituting this expression into (\ref{sootn}), we get
\[
F_n^{(v)}(m+v+1)(m+1)!(m+v+1)!=F_{n+1}^{(v)}(m+v)m!(m+v)! 
+m(m+v)F_n^{(v)}(m+v)m!(m+v)!. 
\]
It is easy  to make sure that, the last relation holds by virtue of (\ref{recurels}).
\end{proof} 
The relation  (\ref{mnogoch}) implies
\[
F^{(v)}_n(m+v)=\frac{1}{(m+v)!}\sum_{j=0}^{n}(-1)^{n+j}T^{(v)}(n, j)j!\left(j+m+v-1\right)!.
\]
Comparing this formula with (\ref{general}), we get
\begin{proposition}
\label{pr:9}
The relation
\begin{equation}
\sum_{j=0}^m(-1)^{m+j}t^{(v)}(m, j)g_{n+j}^{(v)}=\frac{m!}{v!}\sum_{j=0}^{n}(-1)^{n+j}T^{(v)}(n, j)j!\left(j+m+v-1\right)!
\label{sootn123}
\end{equation}
is an identity for all $m, v\geq 0$ and $n\geq 1$.
\end{proposition}
Obviously, the formula (\ref{sootn123}) generalizes the identities (\ref{Fnom121234}) and (\ref{ident}) for the Genocchi numbers and we can consider a countable class of integer sequences $(g_{n}^{(v) })_{v\geq 0}$. Unfortunately, the combinatorial meaning of these numbers for integers $v\geq 2$ is not yet known.
\begin{remark}
Let us remark that the propositions \ref{pr:8} and \ref{pr:9} by themselves do not claim that $g_{n}^{(0)}$ and $g_{n}^{(1)}$ are the Genocchi numbers. The fact that this is really so we get from what is proved in the works \cite{Carlitz2}, \cite{Riordan}, \cite {Dumont}.
\end{remark}
The following table shows some numbers $g_n^{(v)}$ for small values of $n$ and $v$:
\begin{center}
\begin{tabular}{|c|c|c|c|c|c|c|c|c|}
\hline
$n$&$1$&$2$&$3$&$4$&$5$&$6$&$7$ \\
\hline
$g_{n}^{(0)}$&$1$&1&3&17& 155&2073 &38227 \\
\hline
$g_{n}^{(1)}$&$1$&2&8&56& 608& 9440 &198272 \\
\hline
$g_{n}^{(2)}$&$1$&3&15&123& 1527& 26691 & 623391\\
\hline
$g_{n}^{(3)}$&$1$&4&24&224 & 3104& 59904 & 1531264 \\
\hline
$g_{n}^{(4)}$&$1$&5& 35 &365 & 5555& 116765 & 3229475 \\
\hline
$g_{n}^{(5)}$&$1$& 6& 48 & 552 & 9120 & 206688 & 6131712 \\
\hline
$g_{n}^{(6)}$&$1$& 7 & 63 & 791 & 14063 & 340935 & 10774687 \\
\hline
\end{tabular}
\end{center}
The first few polynomials $g_{n}(v)$ have the following form:
\[
g_{1}(v)=1,\;\;
g_{2}(v)=v+1,\;\;
g_{3}(v)=(v+1)\left(v+3\right),\;\;
g_{4}(v)=(v+1)\left(v^2+10v+17\right),
\] 
\[
g_{5}(v)=(v+1)\left(v^3+25v^2+123v+155\right),\;\;
g_{6}(v)=(v+1)\left(v^4+56v^3+590v^2+2000v+2073\right),
\]
\[
g_{7}(v)=(v+1)\left(v^5+119v^4+2362v^3+15942v^2+42485v+38227\right),\ldots
\]
\begin{proposition}
The polynomials $g_{n}(v)$, for all $n\geq 2$, are divisible by $v+1$.
\end{proposition}
\begin{proof}
Substituting $z=v$ into the recurrence relation (\ref{recurels}), we obtain
\[
g_{n+1}(v)=(v+1)F_n^{(v)}(v+1)\;\; \forall n\geq 1.
\]
Since $F_n^{(v)}(v+1)$ is some polynomial in $v$, the proposition should be considered proved.
\end{proof}

\end{document}